\documentclass[11pt,letterpaper]{amsart}

\usepackage{amsmath,amssymb,amsthm,mathtools}
\usepackage{xcolor}
\usepackage[colorlinks=true,linkcolor=blue!55!black,citecolor=blue!55!black,urlcolor=blue!55!black]{hyperref}
\hypersetup{
  pdftitle={Split Normal Bundles and Rigidity of Holomorphic Immersions from Compact Ball Quotients},
  pdfauthor={Rong Du},
  pdfsubject={Holomorphic immersions of compact complex hyperbolic space forms},
  pdfkeywords={complex hyperbolic space form, holomorphic immersion, normal bundle, Higgs bundle, total geodesy}
}

\newtheorem{theorem}{Theorem}[section]
\newtheorem{proposition}[theorem]{Proposition}
\newtheorem{lemma}[theorem]{Lemma}
\newtheorem{corollary}[theorem]{Corollary}
\theoremstyle{remark}
\newtheorem{remark}[theorem]{Remark}

\newcommand{\BB}{\mathbb B}

\newcommand{\OO}{\mathcal O}
\newcommand{\rk}{\operatorname{rk}}
\newcommand{\Ric}{\operatorname{Ric}}

\newcommand{\ch}{\operatorname{ch}}

\title[Split Normal Bundles and Rigidity]{Split Normal Bundles and Rigidity of Holomorphic Immersions
from Compact Ball Quotients}
\author{Rong Du}
\address{School of Mathematical Sciences, Key Laboratory of MEA (Ministry of Education)
and Shanghai Key Laboratory of PMMP, East China Normal University,
Rm.~312, Math. Bldg, No.~500, Dongchuan Road, Shanghai, 200241, P.~R. China.}
\email{rdu@math.ecnu.edu.cn}
\thanks{The author is sponsored by National Natural Science Foundation of China (Grant No.~12471040) and Science and Technology Commission of Shanghai Municipality (Grant No.~22DZ2229014).}
\thanks{The author used DeepSeek for language polishing, grammar correction, and structural revision during the preparation of this manuscript. The author reviewed and edited all AI-generated output and takes full responsibility for the content of the publication.}
\makeatletter
\@ifundefined{subjclassname@2020}{%
  \@namedef{subjclassname@2020}{\textup{2020} Mathematics Subject Classification}%
}{}
\makeatother
\date{}

\begin{document}

\begin{abstract}
Let $f$ be a holomorphic immersion from a compact complex hyperbolic space
form into a higher-dimensional complex hyperbolic space form.  We prove that
if the normal bundle decomposes as a direct sum of holomorphic line bundles,
then the immersion is totally geodesic.  This result has no restriction on
the codimension and extends, under the splitting hypothesis, the classical
rigidity theorem of Cao--Mok.  The proof uses the polarized complex variation
of Hodge structure induced from the target.  The line-bundle decomposition
produces nef zero-Higgs line quotients, and a Chern-character identity forces
the Gauss-defect form to have nontrivial nullity at every point.  The
null-foliation and boundary argument of Cao--Mok then shows that the defect
vanishes identically.  In particular, this answers a question posed by Mok
almost twenty years ago about decomposable normal bundles for complex surfaces immersed in
four-dimensional complex hyperbolic space forms.
\end{abstract}

\maketitle

\section{Introduction}
The immersion problem for smooth ball quotients has long been of concern to mathematicians.
Let
\[
 X=\BB^n/\Gamma,\qquad M=\BB^m/\Lambda
\]
be smooth complex hyperbolic space forms, and assume that $X$ is compact.
For a holomorphic immersion $f:X\to M$, write
\begin{equation}\label{eq:tangent-sequence}
 0\longrightarrow T_X\xrightarrow{\,df\,}f^*T_M
 \longrightarrow N_f\longrightarrow0
\end{equation}
for its tangent sequence.  We call $N_f$ \emph{line-split} if it is a direct
sum of holomorphic line bundles.

Cao and Mok proved that every such immersion is totally geodesic when
$m\leq2n-1$ \cite{CaoMok1990}.  To extended their result to proper
immersions between finite-volume complex hyperbolic space forms under the
same dimension bound \cite{To1993}.  These results form part of the broader
rigidity theory of locally symmetric spaces.  Mok, Siu, and Yeung developed
geometric superrigidity using Bochner identities for harmonic maps
\cite{MokSiuYeung1993}, and Yeung later gave a unified tensorial approach to
vanishing theorems and rigidity for locally symmetric manifolds
\cite{Yeung2001}.  Those works provide important methodological context, but
do not settle the rank-one higher-codimension immersion problem considered
here.  In that range, no non-totally-geodesic immersion with compact source
is known to the author.  In the first open case $n=2$, $m=4$, Mok asked
whether the rank-two normal bundle of a hypothetical non-totally-geodesic
immersion must be indecomposable \cite[Question~4]{Mok2005}.

The following theorem treats every dimension and every codimension at once.

\begin{theorem}\label{thm:main}
Let $n\geq2$ and $m>n$.  Let
\[
 f:X=\BB^n/\Gamma\longrightarrow M=\BB^m/\Lambda
\]
be a holomorphic immersion, where $X$ is compact and both quotients are
smooth.  If
\begin{equation}\label{eq:normal-splitting}
 N_f\simeq\bigoplus_{i=1}^{m-n}L_i
\end{equation}
for holomorphic line bundles $L_i$, then $f$ is totally geodesic.  With the
ball metrics normalized to have holomorphic sectional curvature $-2$, it is
an isometric totally geodesic immersion.
\end{theorem}

\begin{corollary}\label{cor:mok}
Let $f:X^2\to M^4$ be a holomorphic immersion between complex hyperbolic
space forms, with $X$ compact.  If $N_f$ is decomposable, then $f$ is totally
geodesic.  Thus every hypothetical non-totally-geodesic immersion of this
type has indecomposable normal bundle.
\end{corollary}

\begin{proof}
A decomposable rank-two bundle is a direct sum of two line bundles, so the
claim is the case $(n,m)=(2,4)$ of Theorem~\ref{thm:main}.
\end{proof}

The point of the proof is that the splitting hypothesis produces positivity
in exactly the classes that measure the failure of total geodesy.  Put
\[
 K=c_1(K_X),\qquad D=f^*c_1(K_M),\qquad
 \alpha=\frac{K}{n+1}-\frac{D}{m+1}.
\]
The class $\alpha$ is represented by a semipositive Gauss-defect form
$\eta$, and $\eta=0$ precisely when the second fundamental form vanishes.
After a finite \'etale cover, the standard representation of the target
monodromy gives a polarized complex variation of Hodge structure.  Each
summand $L_i$ in \eqref{eq:normal-splitting} yields a zero-Higgs line
quotient $A_i$; restriction to curves shows that $A_i$ is nef.  The
uniformizing Chern identity for $X$ then gives
\[
 \sum_i\exp(c_1(A_i))
 =m+1-(n+1)\exp(-\alpha).
\]
The terms of degrees one and two have opposite positivity signs.  This
forces $\alpha^2K^{n-2}=0$, hence the Gauss defect has a null direction at
every point.  Sections~3--4 of Cao--Mok \cite{CaoMok1990} turn this nullity
into total geodesy; their dimension bound is used only earlier, to obtain
the nullity that is supplied here by the splitting.

\section{The Gauss defect and its nullity}

Normalize every ball metric to have holomorphic sectional curvature $-2$.
If $Y=\BB^r/\Pi$ is a smooth quotient and $\omega_Y$ is its K\"ahler form,
then
\begin{equation}\label{eq:canonical-kahler}
 \left[\frac{\omega_Y}{2\pi}\right]
 =\frac{c_1(K_Y)}{r+1}.
\end{equation}
In particular, $K_X$ is ample and $X$ is projective.

Let $h=f^*g_M$ be the induced K\"ahler metric on $X$, let
$\omega_h=f^*\omega_M$, and let $B$ be the second fundamental form, where
$g_M$ and $\omega_M$ denote the normalized ball metric and its K\"ahler
form.  Define
\begin{equation}\label{eq:defect-form}
 \eta=-\frac{1}{2\pi(n+1)}\Ric(h)
      -\frac{1}{2\pi}\omega_h.
\end{equation}
Here $\Ric(h)$ denotes the Ricci form, with the convention that
$[-\Ric(h)/(2\pi)]=c_1(K_X)$.  In a unitary frame for $h$, the Gauss
equation is
\[
 R_{i\bar j k\bar l}
 =-(h_{i\bar j}h_{k\bar l}+h_{i\bar l}h_{k\bar j})
  -\sum_{A\in\{1,2,\dots, m-n\}} B^A_{ik}\overline{B^A_{jl}}
\]
which yields
$$R_{i\bar{j}} = -(n+1) h_{i\bar{j}} - \sum_{A, k} B_{ik}^A \overline{B_{jk}^A}.$$
We can rewrite this tensor equation in terms of the Ricci form $\Ric(h)$ and the K\"{a}hler form $\omega_h$:
$$\Ric(h) = -(n+1)\omega_h - \sqrt{-1} \sum_{A,i,j,k} B^A_{ik}\overline{B^A_{jk}} \, dz^i \wedge d\bar z^j.$$
Equation (\ref{eq:defect-form}) defines the Gauss-defect form as
\begin{align}
\eta &= -\frac{1}{2\pi(n+1)}\Ric(h)
      -\frac{1}{2\pi}\omega_h. \notag\\
&= -\frac{1}{2\pi(n+1)}\left( -(n+1)\omega_h - \sqrt{-1}\sum_{A,i,j,k} B^A_{ik}\overline{B^A_{jk}} \, dz^i \wedge d\bar z^j \right) - \frac{1}{2\pi}\omega_h\notag \\
&= \frac{\sqrt{-1}}{2\pi(n+1)}
 \sum_{A,i,j,k}B^A_{ik}\overline{B^A_{jk}}
 \,dz^i\wedge d\bar z^j. \label{eq:defect-positive}
\end{align}

Consequently, $\eta$ is closed and semipositive, and
\begin{equation}\label{eq:defect-zero}
 \eta\equiv0\quad\Longleftrightarrow\quad B\equiv0
 \quad\Longleftrightarrow\quad f\text{ is totally geodesic}.
\end{equation}
Equations \eqref{eq:canonical-kahler} and \eqref{eq:defect-form} show that
\begin{equation}\label{eq:defect-class}
 [\eta]=\alpha:=\frac{K}{n+1}-\frac{D}{m+1}.
\end{equation}

We record the geometric input that will be applied after the intersection
calculation.
\begin{lemma}[Cao--Mok nullity principle]\label{lem:nullity}
Let $n\geq2$, and let $f:X^n\to M^m$ be a holomorphic immersion from a
compact complex hyperbolic space form of complex dimension $n$ to a complex hyperbolic space form of complex dimension $m$.
If the Gauss-defect form $\eta$ has a nonzero null vector at every point of
$X$, then $\eta\equiv0$ on $X$ and $f$ is totally geodesic.
\end{lemma}

\begin{proof}
By \eqref{eq:defect-positive}, the null space of $\eta_x$ is the relative
nullity space
\[
 \mathcal N_x=\{\xi\in T_{X,x}:B_x(\xi,u)=0\text{ for every }u\in T_{X,x}\}.
\]
On an open set where its dimension is minimal, the curvature calculation of
Cao--Mok shows that this distribution is integrable, that its leaves are
complex and totally geodesic in both $X$ and $M$, and that the resulting
foliation is holomorphic
\cite[Propositions~3.1--3.2]{CaoMok1990}.  Their boundary argument then
shows that the presence of a positive-dimensional null leaf forces the
Gauss-defect form to vanish on a nonempty open set, and hence everywhere by
real analyticity \cite[Section~4]{CaoMok1990}.  The hypothesis
$m\leq2n-1$ in their theorem is used in Section~2 only to produce a null
direction; Propositions~3.1--3.2 and Section~4 require no upper bound on
$m$.  Thus the assumed pointwise nullity gives $\eta\equiv0$, and
\eqref{eq:defect-zero} completes the proof.
\end{proof}

\section{The uniformizing Higgs bundle and Nef Higgs quotients from a split normal bundle}
The standard representation of $SU(m,1)$ carries a Hermitian form of signature $(m,1)$. In $\mathbb{C}^{m+1}$, consider all one-dimensional complex subspaces passing through the origin. If the squared norm of a non-zero vector on such a complex line with respect to the aforementioned Hermitian form is negative (i.e., $\sum_{j=1}^m |z_j|^2 < |z_{m+1}|^2$), it is called a ``negative line''. If we collect all these negative lines and view them in the complex projective space $\mathbb{CP}^m$, they precisely form an open set, namely the $m$-dimensional complex unit ball $\mathbb{B}^m$. Here, because a negative line uniquely determines its $m$-dimensional orthogonal complement, choosing a negative line is equivalent to specifying a polarized complex Hodge structure. Consequently, $\mathbb{B}^m$ itself serves as the period domain for these Hodge structures.

The immersion of the base manifold $X$ into the ball quotient manifold $M = \mathbb{B}^m/\Lambda$ induces a homomorphism from the fundamental group $\pi_1(X)$ to the ball's isometry group $PU(m,1)$ (i.e., the projective monodromy representation). To perform linear algebra on vector bundles, we lift the monodromy representation from $PU(m,1)$ to the linear matrix group $SU(m,1)$ acting on $\mathbb{C}^{m+1}$ via Lemma~\ref{lem:finite-lift}. This $SU(m,1)$ representation induces a flat complex vector bundle $E$ of rank $m+1$ over $X$. As a point moves on the manifold, the negative lines vary holomorphically, defining a holomorphic subbundle $W \subset E$ of rank $1$. The quotient bundle $V := E/W$ is a holomorphic vector bundle of rank $m$. Smoothly, $E$ decomposes orthogonally as $W \oplus W^\perp$ with $W^\perp \simeq V$. This structure, equipped with the flat connection, the holomorphic subbundle $W$, and the pseudo-Hermitian form serving as the polarization, defines a polarized complex variation of Hodge structure (PVHS) of weight one, whose associated graded Higgs bundle is $E = V \oplus W$.

In general Hodge theory, as a point on the manifold moves, the induced Hodge structure (i.e., the period map into the period domain) cannot vary arbitrarily. Its differential is strictly constrained to map each step of the Hodge filtration only into the ``adjacent'' one; the tangent directions satisfying this constraint constitute the horizontal tangent space of the period domain. However, the Hodge structure here is of weight one. Its filtration consists only of $W$ and the ambient bundle $\mathbb{E}$. Without any intermediate filtration steps, the adjacent transition is simply the map from the rank-one subbundle $W$ to the rank-$m$ quotient bundle $V$. In this specific period domain $\mathbb{B}^m$, the horizontal tangent space coincides with the full tangent space of the period domain. Consequently, any holomorphic map into $\mathbb{B}^m$ satisfies Griffiths transversality automatically and unconditionally, without requiring any additional differential constraints.

%
We first justify passage to such a lift.

\begin{lemma}\label{lem:finite-lift}
There is a finite \'etale cover $p:X'\to X$ such that the projective
monodromy of $f\circ p$ lifts from $PU(m,1)$ to $SU(m,1)$.
\end{lemma}

\begin{proof}
Let $\rho:\pi_1(X)\to PU(m,1)$ be the homomorphism induced by $f$.  Pull
back the central extension
\[
 1\longrightarrow\mu_{m+1}\longrightarrow SU(m,1)
 \longrightarrow PU(m,1)\longrightarrow1
\]
along $\rho$.  The resulting group $\widehat\Gamma$ is a finitely generated
linear group: it embeds in the product of a linear realization of
$\pi_1(X)<PU(n,1)$ and $SU(m,1)$.  By Selberg's lemma
\cite{Alperin1987}, $\widehat\Gamma$ has a torsion-free subgroup
$\widehat\Gamma_0$ of finite index.  Since $\mu_{m+1}$ is finite,
$\widehat\Gamma_0\cap\mu_{m+1}=\{1\}$.  Projection identifies
$\widehat\Gamma_0$ with a finite-index subgroup of $\pi_1(X)$, and the
second projection to $SU(m,1)$ is the required lift over the associated
finite \'etale cover.
\end{proof}

Both line-splitting and total geodesy are preserved under finite \'etale
pullback, and total geodesy descends.  We may therefore pass to the cover in
Lemma~\ref{lem:finite-lift} and keep the notation $X$ and $f$.

Let $\mathbb E$ be the flat rank-$(m+1)$ holomorphic bundle defined by the
lifted standard representation.  The equivariant lift of $f$ to the balls
gives a polarized complex variation of Hodge structure whose Hodge line is
the tautological negative line $W\subset\mathbb E$.  Put
$V=\mathbb E/W$.  Its associated graded
Higgs bundle is
\begin{equation}\label{eq:hodge-bundle}
 E=V\oplus W,\qquad
 \rk V=m,\quad\rk W=1.
\end{equation}

\begin{lemma}\label{lem:tangent-period-domain}
There is a canonical holomorphic isomorphism
\begin{equation} \label{eq:period-tangent}
f^*T_M \simeq \operatorname{Hom}(W, V) \simeq W^* \otimes V.
\end{equation}
\end{lemma}

\begin{proof}
The period domain $\mathbb{B}^m$ consists of all one-dimensional negative-definite subspaces (negative lines) in the pseudo-Hermitian space $\mathbb{C}^{m,1}$, which embeds as an open domain in the complex projective space $\mathbb{CP}^m \simeq \operatorname{Gr}(1, \mathbb{C}^{m+1})$. For any one-dimensional subspace $L \subset \mathbb{C}^{m+1}$, the holomorphic tangent space of the projective space at $[L]$ is canonically isomorphic to $\operatorname{Hom}(L, \mathbb{C}^{m+1}/L)$.
Since $\mathbb{B}^m \subset \mathbb{CP}^m$ is an open complex submanifold, this identification naturally restricts to the tangent bundle of $\mathbb{B}^m$.

Pulling this structure back to $X$ via the immersion $f$, the tautological
line corresponds to $W\subset\mathbb E$, while the quotient corresponds to
$V=\mathbb E/W$.  Hence
\[
f^*T_M \simeq \operatorname{Hom}(W, V) \simeq W^* \otimes V. \qedhere
\]
\end{proof}

Writing $w=c_1(W)$ and taking first Chern classes, the flatness of
$\mathbb E$ and \eqref{eq:period-tangent} give
\begin{equation}\label{eq:w-class}
 -D=c_1(f^*T_M)=c_1(V)-mw=-(m+1)w,
 \qquad w=\frac{D}{m+1}.
\end{equation}

The Higgs field and its contraction with tangent vectors have the form
\[
 \theta:W\longrightarrow V\otimes\Omega_X^1,
 \quad \theta|_V=0,
 \qquad
 \beta:W\otimes T_X\longrightarrow V.
\]
Here $\beta(w\otimes\xi)=df(\xi)(w)$.  Since $f$ is an immersion,
$\beta$ is a subbundle injection.  Tensoring
\eqref{eq:tangent-sequence} by $W$ and using
\eqref{eq:period-tangent}, we obtain
\begin{equation}\label{eq:higgs-normal}
 0\longrightarrow W\otimes T_X\xrightarrow{\,\beta\,}V
 \longrightarrow W\otimes N_f\longrightarrow0.
\end{equation}

Assume now that $N_f=\bigoplus_{i=1}^{m-n}L_i$, and put
\begin{equation}\label{eq:ai}
 A_i=W\otimes L_i,\qquad a_i=c_1(A_i).
\end{equation}
The quotient in \eqref{eq:higgs-normal}, followed by projection onto a
summand, gives
\[
 q_i:E=V\oplus W\longrightarrow A_i,
 \qquad q_i|_W=0.
\]
This is a Higgs quotient with zero Higgs field.  Indeed, $\theta(V)=0$,
while every coefficient of $\theta(W)$ lies in
$\beta(W\otimes T_X)$, which $q_i$ kills.

\begin{proposition}\label{prop:nef}
Every line bundle $A_i$ in \eqref{eq:ai} is nef.
\end{proposition}

\begin{proof}
Let $C\subset X$ be an irreducible curve, and let
$\nu:\widetilde C\to C\hookrightarrow X$ be its normalization.  Pullback of
a polarized complex variation of Hodge structure is again such a
variation.  Its Higgs bundle on the smooth projective curve
$\widetilde C$ is polystable of degree zero
\cite[Theorem~1 and Corollary~1.3]{Simpson1992}.  The map $\nu^*q_i$ is a Higgs line quotient,
so semistability gives
\[
 \deg_{\widetilde C}\nu^*A_i\geq0.
\]
This holds for every irreducible curve in the projective manifold $X$;
hence $A_i$ is nef.
\end{proof}

\section{The Chern identity and proof of the theorem}

The proportionality calculation for a compact ball quotient gives
\begin{equation}\label{eq:chern-polynomial}
 c(T_X)=\left(1-\frac{K}{n+1}\right)^{n+1}
 \quad\text{in }H^*(X,\mathbb R),
\end{equation}
where terms above degree $2n$ are discarded.  Newton's identities applied
to \eqref{eq:chern-polynomial} are equivalently expressed as
\begin{equation}\label{eq:source-ch}
 \ch(\OO_X\oplus T_X)
 =(n+1)\exp\left(-\frac{K}{n+1}\right).
\end{equation}

Since $E=V\oplus W$ is the associated graded of the flat bundle
$\mathbb E$, its positive-degree Chern character vanishes.  Thus
$\ch(V)=m+1-\ch(W)$.  In topological $K$-theory,
\eqref{eq:higgs-normal} and \eqref{eq:source-ch} give
\begin{align}
 \ch(W\otimes N_f)
 &=\ch(V)-\ch(W\otimes T_X)\notag\\
 &=m+1-\ch(W)\ch(\OO_X\oplus T_X)\notag\\
 &=m+1-(n+1)
   \exp\left(\frac{D}{m+1}-\frac{K}{n+1}\right)\notag\\
 &=m+1-(n+1)e^{-\alpha}.
\label{eq:master-ch}
\end{align}
On the other hand, the line-bundle decomposition gives
\[
 \ch(W\otimes N_f)=\sum_{i=1}^{m-n}e^{a_i}.
\]
Comparing the degree-two and degree-four components of
\eqref{eq:master-ch}, we obtain the two identities
\begin{equation}\label{eq:chern-identities}
  \sum_{i=1}^{m-n}a_i=(n+1)\alpha,
 \qquad
 \sum_{i=1}^{m-n}a_i^2=-(n+1)\alpha^2.
\end{equation}

\begin{proposition}\label{prop:defect-square}
Under the line-splitting assumption,
\[
 \alpha^2K^{n-2}=0.
\]
\end{proposition}

\begin{proof}
By Proposition~\ref{prop:nef}, every $a_i$ is nef.  The first identity in
\eqref{eq:chern-identities} shows that $\alpha$ is nef as well.  Since $K$
is ample,
\[
 a_i^2K^{n-2}\geq0
 \quad\text{and}\quad
 \alpha^2K^{n-2}\geq0.
\]
Intersecting the second identity in \eqref{eq:chern-identities} with
$K^{n-2}$ gives
\[
 0\leq\sum_i a_i^2K^{n-2}
 =-(n+1)\alpha^2K^{n-2}\leq0.
\]
Both sides must vanish, proving the assertion.
\end{proof}

\begin{proof}[Proof of Theorem~\ref{thm:main}]
We have already passed to the finite \'etale cover supplied by
Lemma~\ref{lem:finite-lift}; line-splitting pulls back to that cover, and
total geodesy descends to the original $X$.  By
Proposition~\ref{prop:defect-square}, $\alpha^2K^{n-2}=0$.  In view of
\eqref{eq:canonical-kahler} and $[\eta]=\alpha$, this is equivalent to
\[
\int_X \eta^2 \wedge \omega_X^{n-2} = 0,
\]
where $\omega_X$ is the K\"ahler form of the intrinsic normalized ball metric on $X$.

At any point $x \in X$, choose a local $\omega_X$-orthonormal $(1,0)$-basis $\{\theta^1, \dots, \theta^n\}$ diagonalizing the real $(1,1)$-form $\eta_x$:
\[
\omega_{X,x} = \sqrt{-1} \sum_{i=1}^n \theta^i \wedge \bar{\theta}^i, \qquad \eta_x = \sqrt{-1} \sum_{i=1}^n \lambda_i \theta^i \wedge \bar{\theta}^i,
\]
where $\lambda_1, \dots, \lambda_n \in \mathbb{R}$ are the eigenvalues of $\eta_x$ with respect to $\omega_{X,x}$. Since $\eta$ is semipositive by (\ref{eq:defect-positive}), every eigenvalue satisfies $\lambda_i \ge 0$. A direct computation of the wedge product gives
\[
\eta^2 \wedge \omega_X^{n-2} = \frac{2}{n(n-1)} \left( \sum_{1 \le i < j \le n} \lambda_i \lambda_j \right) \omega_X^n.
\]
Because $\lambda_i \ge 0$ and the top volume form $\omega_X^n > 0$, the continuous integrand $\eta^2 \wedge \omega_X^{n-2}$ is pointwise non-negative everywhere on $X$. Having zero integral over the compact manifold $X$, the integrand must vanish identically:
\[
\sum_{1 \le i < j \le n} \lambda_i(x) \lambda_j(x) = 0 \quad \text{for every } x \in X.
\]

Since all $\lambda_i(x) \ge 0$, at most one eigenvalue can be strictly positive at each point $x \in X$. Consequently, the pointwise rank satisfies $\operatorname{rk}(\eta_x) \le 1$. The dimension of the null space $\mathcal{N}_x = \ker \eta_x$ therefore satisfies
\[
\dim_{\mathbb{C}} \mathcal{N}_x = n - \operatorname{rk}(\eta_x) \ge n - 1.
\]
Because $n \ge 2$, we have $\dim_{\mathbb{C}} \mathcal{N}_x \ge 1$, which proves that the Gauss-defect form $\eta$ has a nonzero null vector at every point of $X$.

Applying Lemma~\ref{lem:nullity}, we obtain $\eta \equiv 0$ on all of $X$. By (\ref{eq:defect-zero}), this implies that the second fundamental form $B \equiv 0$ identically, so $f$ is totally geodesic on the cover.

The second fundamental form remains zero after descending from the finite cover to the original manifold $X$. Finally, the universal cover lift $\tilde{f}: \mathbb{B}^n \to \mathbb{B}^m$ of a totally geodesic holomorphic immersion maps $\mathbb{B}^n$ onto a totally geodesic complex sub-ball in $\mathbb{B}^m$. The canonical inclusion $\mathbb{B}^n \hookrightarrow \mathbb{B}^m$ preserves the normalized Bergman metric. Thus $f$ is an isometric totally geodesic immersion.
\end{proof}

\begin{remark}\label{rem:range}
When $m\leq2n-1$, the conclusion of Theorem~\ref{thm:main} follows from
Cao--Mok without any hypothesis on $N_f$.  Thus the contribution of the
line-splitting argument lies in the complementary range $m\geq2n$.  The
target need not be compact; compactness is used only for the source.
\end{remark}


\section*{Acknowledgements}
The author would like to thank Professor N. Mok for stimulating discussions
on ball quotients, for bringing related work to his attention, and for
providing an excellent research environment at the University of Hong Kong
during his visit.

\begingroup
\small
\bibliographystyle{plain}
\bibliography{normal_bundle_ball_quotients_high_dim}

\begin{thebibliography}{1}

\bibitem{Alperin1987}
Roger~C. Alperin.
\newblock An elementary account of {Selberg}'s lemma.
\newblock {\em L'Enseignement Math\'{e}matique}, 33(3-4):269--273, 1987.

\bibitem{CaoMok1990}
Huai-Dong Cao and Ngaiming Mok.
\newblock Holomorphic immersions between compact hyperbolic space forms.
\newblock {\em Inventiones Mathematicae}, 100(1):49--61, 1990.

\bibitem{Mok2005}
Ngaiming Mok.
\newblock On holomorphic immersions into {K\"ahler} manifolds of constant
  holomorphic sectional curvature.
\newblock {\em Science in China Series A: Mathematics}, 48(Supplement
  1):123--145, December 2005.

\bibitem{MokSiuYeung1993}
Ngaiming Mok, Yum-Tong Siu, and Sai-Kee Yeung.
\newblock Geometric superrigidity.
\newblock {\em Inventiones Mathematicae}, 113(1):57--83, 1993.

\bibitem{Simpson1992}
Carlos~T. Simpson.
\newblock Higgs bundles and local systems.
\newblock {\em Publications Math\'{e}matiques de l'IH\'{E}S}, 75:5--95, 1992.

\bibitem{To1993}
Wing-Keung To.
\newblock Total geodesy of proper holomorphic immersions between complex
  hyperbolic space forms of finite volume.
\newblock {\em Mathematische Annalen}, 297(1):59--84, 1993.

\bibitem{Yeung2001}
Sai-Kee Yeung.
\newblock On vanishing theorems and rigidity of locally symmetric manifolds.
\newblock {\em Geometric and Functional Analysis}, 11(1):175--198, 2001.

\end{thebibliography}
\endgroup

\end{document}